\documentclass{article}
\usepackage{times}
\usepackage{amsmath,amssymb}
\newtheorem{theorem}{Theorem}[section]

\newtheorem{lemma}[theorem]{Lemma}
\newtheorem{definition}[theorem]{Definition}

\newtheorem{proposition}[theorem]{Proposition}
\newcommand{\qed}{\rule{1mm}{3mm}}     
\newenvironment{proof}{\vspace*{\parsep}\noindent {\bf proof:}}{\qed\\[1em]}
\begin{document}

\title{On the Cauchy Completeness of the Constructive Cauchy Reals}

\author{Robert S. Lubarsky \\ Dept. of Mathematical Sciences
\\ Florida Atlantic University \\ Boca Raton, FL 33431 \\
Robert.Lubarsky@alum.mit.edu}
\maketitle

\begin{abstract}
It is consistent with constructive set theory (without Countable
Choice, clearly) that the Cauchy reals (equivalence classes of
Cauchy sequences of rationals) are not Cauchy complete. Related
results are also shown, such as that a Cauchy sequence of
rationals may not have a modulus of convergence, and that a Cauchy
sequence of Cauchy sequences may not converge to a Cauchy
sequence, among others.
\end{abstract}

\section{Introduction}
Are the reals Cauchy complete? This is, for every Cauchy sequence
of real numbers, is there a real number which is its limit?

This sounds as though the answer should be ``of course". After
all, the reals are defined pretty much to make this true. The
reason to move from the rationals to the reals is exactly to
``fill in the holes" that the rationals have. So however you
define ${\mathbb R}$, you'd think its Cauchy completeness would be
immediate. At the very least, this property would be a litmus test
for any proffered definition.

In fact, for the two most common notions of real number, Dedekind
and Cauchy real, this is indeed the case, under classical logic. First off,
classically Cauchy and Dedekind reals are equivalent anyway. Then,
taking a real as an equivalence class of Cauchy sequences,
given any Cauchy sequence of reals, a canonical representative can
be chosen from each real, and a limit real can be built from them
by a kind of diagonalization, all pretty easily.

Constructively, though, this whole procedure breaks down. Starting
even at the beginning, Cauchy and Dedekind reals are no longer
equivalent notions (see \cite {FH} or \cite{LR}). While the
Dedekind reals are complete, working with the Cauchy reals, it's
not clear that a representative can be chosen from each
equivalence class, or, even if you could, that a limit could be
built from them by any means.

It is the purpose of this paper to show that, indeed, such
constructions are not in general possible, answering a question of
Martin Escardo and Alex Simpson \cite{ES}.

While the original motivation of this work was to show the final
theorem, that the Cauchy reals are not provably Cauchy complete,
it is instructive to lay out the framework and examine the related
questions so laid bare. A Cauchy real is understood as an
equivalence class of Cauchy sequences of rationals. When working
with a Cauchy sequence, one usually needs to know not only that
the sequence converges, but also how fast. In classical set
theory, this is definable from the sequence itself, and so is not
problematic. The same cannot be said for other contexts. For
instance, in recursion theory, the complexity of the convergence
rate might be important. In our context, constructive (a.k.a.
intuitionistic) set theory, the standard way to define a modulus
of convergence just doesn't work. Certainly given a Cauchy
sequence $X(n)$ and positive rational $\epsilon$, there is an
integer $N$ such that for $m, n > N \; \mid X(n) - X(m) \mid <
\epsilon$: that's the definition of a Cauchy sequence. A modulus
of convergence is a function $f$ such that for any such $\epsilon$
$f$ on $\epsilon$ returns such an $N$. Classically this is easy:
let $f$ return the least such $N$. Intuitionistically that won't
work. And there's no obvious alternative. So a real is taken to be
an equivalence class of pairs $\langle X, f \rangle$, where $X$ is
a Cauchy sequence and $f$ a modulus of convergence.

One immediate source of confusion here is identifying reals with
sequence-modulus pairs. A real is an equivalence class of such
pairs, and it is not obvious how a representative can be chosen
constructively from each real; in fact, this cannot in general be
done, as we shall see. This distinction has not always be made
though. For instance, as observed by Fred Richman, in \cite{TvD},
the Cauchy completeness of the reals was stated as a theorem, but
what was proved was the Cauchy completeness of sequence-modulus
pairs. To be precise, what was shown was that, given a countable sequence,
with its own modulus of convergence, of sequence-modulus pairs,
then there is a limit sequence, with modulus. For that matter, it
is not hard (and left to the reader) that, even if the given
sequence does not come
equipped with its own modulus, it still has a Cauchy sequence as a
limit, although we will have to punt on the limit having a
modulus. But neither of those two observations is the Cauchy
completeness of the reals.

Nonetheless, these observations open up the topic about what kinds
of behavior in the limit one can expect given certain input data.
There are two, independent parameters. Does the outside Cauchy
sequence have a modulus of convergence? And are its individual
members sequence-modulus pairs, or merely naked sequences?
Notice that, while the first question is
yes-no, the second has a middle option: the sequences have
moduli, but not uniformly.
Perhaps each entry in the big sequence is simply a Cauchy sequence
of rationals, and it is hypothesized
to have a modulus somewhere, with no information about the modulus given.
These possibilities are all summarized in the following
table.
\begin{center} Entries in the outside Cauchy sequence are:
\begin{tabular}{c|c|c|c}
  & seq-mod pairs &  seqs that have mods & seqs
  that may not \\
  & & somewhere & have mods anywhere\\
  \hline
 Outside seq & & &\\
 has a mod &  &  &  \\
 \hline
 Outside seq & & &\\
 doesn't & & &\\
 have a mod &  &  &  \\
\end{tabular}\\
\end{center}

Not to put the cart before the
horse, is it even possible to have a Cauchy sequence with no
modulus of convergence, or with only
non-uniform moduli? It has already been observed that the
obvious classical definition of such does not work
intuitionistically, but it still remains to be shown that no such
definition is possible.

The goal of this paper is to prove the negative results as much as
possible, that any given hypothesis does not show that there is a
limit Cauchy sequence, or in certain cases that there is no limit
with a modulus.

The positive results are all easy enough and so are left as
exercises. They are:

\begin{enumerate}
\item Every Cauchy sequence with modulus of sequence-modulus pairs
has a limit sequence with modulus.

\item Every Cauchy sequence of sequence-modulus pairs has a limit
sequence.

\item Every Cauchy sequence with modulus of Cauchy sequences has a
limit sequence.

\end{enumerate}

In tabular form, the positive results are:
\begin{center} Entries in the outside Cauchy sequence are:\\

\begin{tabular}{c|c|c|c}
  & seq-mod pairs &  seqs that have mods & seqs
  that may not \\
  & & somewhere & have mods anywhere\\
  \hline
 Outside seq &There is a &There is a &There is a\\
 has a mod &limit with mod.  &limit sequence.  & limit sequence. \\
 \hline
 Outside seq & & & \\
 doesn't &There is a &  & \\
 have a mod & limit sequence. &   & \\
\end{tabular}\\
\end{center}

Regarding the first and last columns, the negative results are
that these are the positive results cited above are the best
possible. In detail:\\

{\bf Theorem \ref{thm1}}
{\it IZF$_{Ref}$ does not prove that every Cauchy sequence has a modulus of
convergence. It follows that IZF$_{Ref}$ does not prove that
every Cauchy sequence of
sequence-modulus pairs converges to a Cauchy sequence with a
modulus of convergence.}\\

{\bf Theorem \ref{thm2}}
{\it IZF$_{Ref}$ does not prove that every Cauchy sequence of Cauchy
sequences converges to a Cauchy sequence.}\\

{\bf Theorem \ref{thm3}}
{\it IZF$_{Ref}$ does not prove that every Cauchy sequence with a modulus
of convergence of Cauchy sequences converges to a Cauchy sequence
with a modulus of convergence.}\\

The middle column is discussed briefly in the questions in the
last section of this paper.

In tabular form, these negative results are:

\begin{center} Entries in the outside Cauchy sequence are:
\nopagebreak

\begin{tabular}{c|c|c|c}
  & seq-mod pairs &  seqs that have mods & seqs
  that may not \\
  & & somewhere & have mods anywhere\\
  \hline
 Outside seq & &See questions, &Limit may not have\\
 has a mod &  &section 7.  & a mod. thm \ref{thm3} \\
 \hline
 Outside seq &Limit may not & &There may not \\
 doesn't &have a mod. &See questions, & even be a limit.\\
 have a mod & thm \ref{thm1} & section 7. & thm \ref{thm2}\\
\end{tabular}\\
\end{center}

Then there is the major negative result, the original and ultimate
motivation of this work:\\

{\bf Theorem \ref{thm4}}
{\it IZF$_{Ref}$ does not prove that every Cauchy sequence of reals
has a limit.}\\

Recalling that a real is here taken as an equivalence class of
sequence-modulus pairs, to prove this result it would suffice to
construct a Cauchy sequence (perhaps itself without modulus) of
reals, with no sequence-modulus pair as a limit. We will do a tad
better, constructing a Cauchy sequence, with modulus, of reals, with
no Cauchy sequence, even without modulus, as a limit.

At this point a word about the meta-theory is in order. The
results here are stated as non-theorems of IZF$_{Ref}$, which is
the variant of IZF in which the Collection schema is replaced by
the Reflection schema. The point is that these independence
results are not meant to be based on a weakness of the underlying
set theory. Hence the set theory taken is the strongest version of
the intuitionistic theories commonly considered. The results would
remain valid if IZF were augmented by yet stronger hypotheses,
such as large cardinals. Of course, these remarks do not apply if
IZF is augmented by whatever choice principle would be enough to
build the sequences and moduli here claimed not to exist. Clearly
Dependent Choice is strong enough for everything at issue here:
choosing representatives from equivalence classes, building
moduli, constructing Cauchy sequences. The question what weaker
choice principle/s would suffice is addressed in \cite{LRi}.

Regarding the methodology, counter-examples are constructed in
each case. These examples could be presented as either topological
or Kripke models. While certain relations among topological and
Kripke models are known, it is not clear to the author that the
natural models in the cases before us are really the same. While
the family resemblance is unmistakable (compare sections 2 and 3
below), it is, for example, at best nonobvious where the
non-standard integers in the Kripke models are hidden in the
topological ones. A better understanding of the relations among
Kripke and topological models would be a worthwhile project for
some other time. For now, we would like to present the reader with
adequate information without being long-winded. Hence all of the
constructions will be presented as topological models, since there
is better technology for dealing with them. In particular, there
is already a meta-theorem (see \cite{G}) that the (full) model
over any Heyting algebra models IZF (easily, IZF$_{Ref}$ too). So
we will never have to prove that our topological models satisfy
IZF$_{Ref}$. In contrast, we know of no such meta-theorem that
would apply to the Kripke models in question. In the simplest
case, the first theorem, the Kripke model will also be given, so
the reader can see what's going on there. But even a cursory
glance at that argument should make it clear why the author does
not want to repeat the proofs of IZF$_{Ref}$ and all the auxiliary
lemmas, and the reader likely does not want to read them, three
more times.

One last word about notation/terminology. For p an open set in a
topological space and $\phi$ a formula in set theory (possibly
with parameters from the topological model), ``$p \subseteq \|
\phi \|", ``p \Vdash \phi$", and ``$p$ forces $\phi$" all mean the
same thing. Also, ``WLOG" stands for ``without loss of
generality".

\section{Not every Cauchy sequence has a modulus of convergence}

\begin{theorem}
\label{thm1} IZF$_{Ref}$ does not prove that every Cauchy sequence
has a modulus of convergence. It follows that IZF$_{Ref}$ does not
prove that every Cauchy sequence of sequence-modulus pairs
converges to a Cauchy sequence with a modulus of convergence.
\end{theorem}

The second assertion follows immediately from the first: given a
Cauchy sequence $X(n)$, for each $n$ let $X_n$ be the constant
sequence $X(n)$ paired with some modulus of convergence
independent of $n$. Sending $X$ to the sequence $\langle X_n \mid
n \in {\bf N} \rangle$ embeds the Cauchy reals into Cauchy
sequences of sequence-modulus pairs. If provably every one of the
latter had a modulus, so would each of the former.

To prove the first assertion, we will build a topological model
with a specific Cauchy sequence $Z(n)$ of rationals with no
modulus of convergence.

The topological space $T$ consists of all Cauchy sequences of
rationals.

A basic open set is given by $(p, I)$, where $p$ is a finite
sequence of rationals and $I$ is an open interval. A Cauchy
sequence $X$ is in (the open set determined by) $(p, I)$ if $p
\subseteq X$, rng($X \backslash p) \subseteq I$, and lim($X$) $\in
I$. (Notice that, under this representation, the whole space $T$
is given by ($\emptyset, {\mathbb R}$), and the empty set is given
by $(p, \emptyset)$ for any $p$.)

For this to generate a topology, it suffices to show that the
basic open sets are closed under intersection. Given $(p, I)$ and
$(q, J)$, if $p$ and $q$ are not compatible (i.e. neither is an
extension of the other), then $(p, I) \cap (q, J) = \emptyset$.
Otherwise WLOG let $q \supseteq p$. If rng($q \backslash p) \not
\subseteq I$ then again $(p, I) \cap (q, J) = \emptyset$.
Otherwise $(p, I) \cap (q, J) = (p \cup q, I \cap J) = (q, I \cap
J)$.

Let $M$ be the Heyting-valued models based on $T$, as describes in
e.g. \cite{G}. Briefly, a set in $M$ is a collection of objects of
the form $\langle \sigma, (p, I) \rangle$, where $\sigma$
inductively is a set in $M$. It is shown in \cite{G} that $M
\models$ IZF$_{Coll}$ (assuming IZF$_{Coll}$ in the meta-theory).
Similarly, assuming IZF$_{Ref}$ in the meta-theory yields $M$
$\models$ IZF$_{Ref}$.

We are interested in the term $\{ \langle {\bar p}, (p, I) \rangle
\mid (p, I)$ is an open set$\}$. (Here ${\bar p}$ is the canonical
name for $p$. Each set in $V$ has a canonical name in $M$ by
choosing $(p, I)$ to be ($\emptyset, {\mathbb R}$) hereditarily:
${\bar x} = \{ \langle {\bar y}, (\emptyset, {\mathbb R}) \rangle
\mid y \in x \}$.) We will call this term $Z$.

\begin{proposition}
$\|$ Z is a Cauchy sequence $\|$ = T.
\end{proposition}
\begin{proof}
To see that $\|Z$ is total$\| = T$, let $N$ be an integer. (Note
that each integer in $M$ can be identified locally with an integer
in $V$. For notational ease, we will identify integers in $M$ and
$V$.) Let $p$ be any sequence of rationals of length $> N$. Then
$(p, {\mathbb R}) \subseteq \| Z(N) = p(N) \| \subseteq \| N \in$
dom($Z) \|$. As $T$ is covered by the open sets of that form, $T$
$\subseteq \| N \in$ dom($Z) \|$. That $Z$ is a function is
similarly easy.

As for $Z$ being Cauchy, again let $N$ be an integer and $X$ be in
$T$. Since $X$ is Cauchy, there is an integer $M$ such that beyond
$M$ $X$ stays within an interval $I$ of size 1/(2$N$). Of course,
$X$'s limit might be an endpoint of $I$. So let $J$ extend $I$ on
either side and still have length less that 1/$N$. Then $X \in (X
\upharpoonright M, J) \subseteq \| \forall m, n > M \mid Z(m) -
Z(n) \mid \; \leq 1/N \|$, making $Z$ ``Cauchy for 1/$N$", to coin
a phrase.
\end{proof}

In order to complete the theorem, we need only prove the following
\begin{proposition}
$\|$Z has no modulus of convergence$\|$ = T.
\end{proposition}
\begin{proof}
Suppose $(p, I) \subseteq \| f$ is a modulus of convergence for
$Z\|$. WLOG $I$ is a finite interval. Let $n$ be such that 1/$n$
is less than the length of $I$, and let $\epsilon$ be (length($I$)
- 1/$n$)/2.

Let $(q, J) \subseteq (p, I)$ force a value $m$ for $f(n)$. WLOG
length($q$) $> m$, as $q$ could be so extended. If $J=I$, then
$(q, J)$ could be extended simply by extending $q$ with two values
a distance greater than 1/$n$ apart, thereby forcing $f$ not to be
a modulus of convergence. So $J \subset I$. That means either inf
$J$ $>$ inf $I$ or sup $J$ $<$ sup $I$. WLOG assume the latter.
Let mid $J$ be the midpoint of $J$, $q_0$ be $q$ extended by mid
$J$, and $J_0$ be (mid $J$, sup $J$). Then ($q_0, J_0) \subseteq
(q, J)$, and therefore ($q_0, J_0) \subseteq \|f(n) = m\|$.

What ($q_1, J_1$) is depends:

CASE I: There is an open set $K$ containing sup($J_0$) such that
($q_0, K) \subseteq \|f(n) = m\|$. Then let $j_{max}$ be the sup
of the right-hand endpoints (i.e. sups) of all such $K$'s. Let
$q_1$ be $q_0$ extended by sup $J_0$.

\underline{Claim}: ($q_1$, (sup $J_0, j_{max})) \subseteq \|f(n) =
m\|$.

\begin{proof}
Let $j \in$ (sup $J_0, j_{max}$). By hypothesis, there is a $K$
such that ($q_1$, (sup $J_0, j)) \subseteq (q_1, K) \subseteq
\|f(n) = m\|$. Since ($q_1$, (sup $J_0, j_{max}$)) is the union of
the various ($q_1$, (sup $J_0, j$))'s over such $j$'s, the claim
follows.
\end{proof}
Of course, $J_1$ will be (sup $J_0, j_{max}$).

CASE II: Not Case I. Then extend by the midpoint again. That is,
$q_1$ is $q_0$ extended by mid $J_0$, and $J_1$ is (mid $J_0$, sup
$J_0$). Also in this case, ($q_1, J_1) \Vdash f(n) = m$.

Clearly we would like to continue this construction. The only
thing that might be a problem is if the right-hand endpoint of
some $J_k$ equals (or goes beyond!) sup $I$, as we need to stay
beneath $(p, I)$. In fact, as soon as sup($J_k) >$ sup($I$) -
$\epsilon$ (if ever), extend $q_k$ by something within $\epsilon$
of sup $I$, and continue the construction with left and right
reversed. That is, instead of going right, we now go left. This is
called ``turning around".

What happens next depends.

CASE A: We turned around, and after finitely many more steps, some
$J_k$ has its inf under inf($I$) + $\epsilon$. Then extend $q_k$
by something within $\epsilon$ of inf $I$. This explicitly blows
$f$ being a modulus of convergence for $Z$.

CASE B: Not Case A. So past a certain point (either the stage at
which we turned around, or, if none, from the beginning) we're
marching monotonically toward one of $I$'s endpoints, but will
always stay at least $\epsilon$ away. WLOG suppose we didn't turn
around. Then the construction will continue for infinitely many
stages. The $q_k$'s so produced will in the limit be a (monotonic
and bounded, hence) Cauchy sequence $X$. Furthermore, lim $X$ is
the limit of the sup($J_k$)'s. Finally, $X \in (p, I)$. Hence
there is an open set $(q', K)$ with $X \in (q', K) \subseteq
\|f(n) = m'\|$, for some $m'$. Let $k$ be such that sup($J_k$)
$\in K$, and $q_k \supseteq q'$. Consider ($q_k, J_k$). Note that
($q_k, J_k \cap K$) extends both ($q_k, J_k$) and $(q', K)$, hence
forces both $f(n) = m$ and $f(n) = m'$, which means that $m=m'$.

Therefore, at this stage in the construction, we are in Case I. By
the construction, $J_{k+1}$ = (sup $J_k, j_{max}$), where $j_{max}
\geq$ sup $K >$ lim $X$ = lim$_k$ (sup($J_k$)) $\geq$
sup($J_{k+1}$) = $j_{max}$, a contradiction.
\end{proof}

\section{Same theorem, Kripke model version}
{\bf Theorem \ref{thm1}}
{\it IZF$_{Ref}$ does not prove that every Cauchy sequence has a modulus of
convergence.}

To repeat the justification given in the introduction, even though
the theorem proved in this section is exactly the same as in the
previous, this argument is being given for methodological
considerations. The prior construction is topological, the coming
one is a Kripke model, and it is not clear (at least to the
author) how one could convert one to the other, either
mechanically via a meta-theorem or with some human insight. Hence
to help develop the technology of Kripke models, this alternate
proof is presented.

\subsection{Construction of the Model}
Let $M_{0} \prec M_{1} \prec$ ... be an $\omega$-sequence of
models of ZF set theory and of elementary embeddings among them,
as indicated, such that the sequence from $M_{n}$ on is definable
in $M_{n}$, and such that each thinks that the next has
non-standard integers. Notice that this is easy to define (mod
getting a model of ZF in the first place): an iterated ultrapower
using any non-principal ultrafilter on $\omega$ will do. We will
ambiguously use the symbol f to stand for any of the elementary
embeddings inherent in the M$_{n}$-sequence.

The Kripke model $M$ will have underlying partial order a
non-rooted tree; the bottom node (level 0) will have continuum (in
the sense of $M_0$) many nodes, and the branching at a node of
level $n$ will be of size continuum in the sense of $M_{n+1}$. (We
will eventually name each node by associating a Cauchy sequence to
it. Some motivation will be presented during this section, and the
final association will be at the end of this section.)
Satisfaction at a node will be indicated with the symbol
$\models$. There is a ground Kripke model, which, at each node of
level $n$, has a copy of $M_{n}$. The transition functions (from a
node to a following node) are the elementary embeddings given with
the original sequence of models (and therefore will be notated by
$f$ again). Note that by the elementarity of the extensions, this
Kripke model is a model of classical ZF. More importantly, the
model restricted to any node of level $n$ is definable in $M_{n}$,
because the original $M$-sequence was so definable.

The final model $M$ will be an extension of the ground model that
will be described like a forcing extension. That is, $M$ will
consist of (equivalence classes of) the terms from the ground
model. The terms are defined at each node separately, inductively
on the ordinals in that model. At any stage $\alpha$, a term of
stage $\alpha$ is a set $\sigma$ of the form $\{ \langle
\sigma_{j}, (p_{j}, I_{j}) \rangle \mid j \in J \}$, where $J$ is
some index set, each $\sigma_{i}$ is a term of stage $< \alpha$,
each $p_{j}$ is a finite function from ${\mathbb N}$ to ${\mathbb
Q}$, and each $I_{j}$ is an open rational interval on the real
line. Note that all sets from the ground model have canonical
names, by choosing each $p_{j}$ to be the empty function and $I_j$
to be the whole real line, hereditarily.

Notice also that the definition of the terms given above will be
interpreted differently at each node of the ground Kripke model,
as the ${\mathbb N}$ and ${\mathbb Q}$ change from node to node.
However, any term at a node gets sent by the transition function
$f$ to a corresponding term at any given later node. The
definitions given later, such as the forcing relation $\Vdash$,
are all interpretable in each $M_{n}$, and coherently so, via the
elementary embeddings.

As a condition, each finite function $p$ is saying ``the Cauchy
sequence includes me", and each interval $I$ is saying ``future
rationals in the Cauchy sequence have to come from me". For each
node of level $n$ there will be an associated Cauchy sequence $r$
(in the sense of $M_{n}$) such that at that node the true $p$'s
and $I$'s will be those compatible with $r$ (or, perhaps, those
with which $r$ is compatible, as the reader will). You might
reasonably think that compatibility means ``$p \subset r$ and
rng($r \backslash p) \subseteq I$": roughly, ``$r$ extends $p$,
and anything in $r$ beyond $p$ comes from $I$". But that's not
quite right. Consider the Cauchy sequence $r(n) = 1/n \; (n \geq$
1). rng$(r) \subseteq$ (0, 2), but in a non-standard extension,
$r$'s pattern could change at a non-standard integer; at that
point, it would be too late for $r$ to change by a standard
amount, but it could change by an infinitesimal amount. So the
range of $r$ could include (infinitely small) negative numbers,
which are outside of (0, 2). Hence we have the following

\begin{definition}
A condition (p, I) and a Cauchy sequence r are \underline{compatible} if
p $\subseteq$ r,  rng(r$\backslash$p) $\subseteq$
I, and lim(r) $\in$ I.

(p, I) is compatible with a finite function q if p $\subseteq$ q
and rng(q$\backslash$p) $\subseteq$ I.
\end{definition}

Given this notion of compatibility, speaking intuitively here, a
term $\sigma$ can be thought of as being interpretable (with
notation $\sigma^{r}$) inductively in $M_{n}$ as $\{
\sigma_{j}^{r} \mid \langle \sigma_{j}, p_j, I_{j} \rangle \in
\sigma$ and $r$ is compatible with ($p_j, I_j) \}$. (This notion
is hidden in the more formal development below, where we define
and then mod out by =$_M$.)

Our next medium-term goal is to define the primitive relations at
each node, =$_{M}$ and $\in_{M}$ (the subscript being used to
prevent confusion with equality and membership of the ambient
models $M_{n}$). In order to do this, we need first to develop our
space's topology.

\begin{definition} (q, J) $\leq$ (p, I) ((q, J) \underline{extends} (p, I)) if
q $\supseteq$ p, J $\subseteq$ I, and rng(q$\backslash$p) $\subset$ I.

{\it C} = $\{ (p_j, I_j) \mid j \in J \}$ \underline{covers} (p, I) if each
(p$_j$, I$_j$) extends (p, I) and each Cauchy sequence r
compatible with (p, I) is compatible with some  (p$_j$, I$_j$).

$\leq$ induces a notion of compatibility of conditions (having a
common extension). We say that a typical member $\langle \sigma,$
(p, I) $\rangle$ of a term is \underline{compatible} with (q, J) if
(p, I) and (q, J) are compatible.
\end{definition}

We need some basic facts about this p.o., starting with the fact
that it is a p.o.
\begin{lemma}
\begin{enumerate}
\item $\leq$ is reflexive, transitive, and anti-symmetric.
\item If (p, I) and (q, J) are each compatible with a Cauchy
sequence r, then they are compatible with each other.
\item If (p, I) and (q, J) are compatible, then their glb in the
p.o. is (p$\cup$q, I$\cap$J).
\item \{(p, I)\} covers (p, I).
\item A cover of a cover is a cover. That is, if {\it C} covers
(p, I), and, for each (p$_j$, I$_j$) $\in$ {\it C}, {\it C}$_j$
covers (p$_j$, I$_j$), then $\bigcup_j${\it C}$_j$ covers (p, I).
\item If {\it C} covers (p, I) and (q, J) $\leq$ (p, I), then (q,
J) is covered by {\it C} $\wedge$ (q, J) =$_{def}$ \{(p$_j\cup$q,
I$_j\cap$J) $\mid$ (q, J) is compatible with (p$_j$, I$_j$) $\in$
{\it C}\}.

\end{enumerate}
\end{lemma}
\begin{proof}
Left to the reader.
\end{proof}

Now we are in a position to define =$_{M}$ and $\in_{M}$.
This will be done via a forcing relation $\Vdash$.

\begin{definition}
$(p, I) \Vdash \sigma =_{M} \tau$ and $(p, I)
\Vdash \sigma \in_{M} \tau$
are defined inductively on $\sigma$ and $\tau$, simultaneously for all
$(p, I)$:

$(p, I) \Vdash \sigma =_{M} \tau$ iff for all
$\langle \sigma_{j}, (p_{j}, I_j) \rangle \in \sigma$ compatible
with (p, I) $(p \cup p_j, I \cap I_{j}) \Vdash \sigma_{j} \in_{M} \tau$ and vice
versa, and

$(p, I) \Vdash \sigma \in_{M} \tau$ iff there is a cover {\it C}
of (p, I) such that for all $(p_j, I_j) \in {\it C}$ there is a
$\langle \tau_{k}, (p_{k}, I_k) \rangle \in \tau$ such that $(p_j,
I_j) \leq (p_k, I_k)$ and $(p_j, I_j) \Vdash \sigma =_{M}
\tau_{k}$.
\end{definition}

(We will later extend this forcing relation to all formulas.)

\begin{definition}
At a node (with associated real r), for any two terms $\sigma$ and $\tau$,
$\sigma =_{M} \tau$ iff, for
some (p, I) compatible with r, (p, I) $\Vdash \sigma =_{M} \tau$.

Also, $\sigma \in_{M} \tau$ iff for
some (p, I) compatible with r, (p, I) $\Vdash \sigma \in_{M} \tau$.
\end{definition}
Thus we have a
first-order structure at each node.

The transition functions are the same as before. That is, if
$\sigma$ is an object at a node, then it's a term, meaning in
particular it's a set in some $M_{n}$. Any later node has for its
universe the terms from some $M_{m}, m \geq n$. With $f$ the
elementary embedding from $M_{n}$ to $M_{m}$, $f$ can also serve
as the transition function between the given nodes. These
transition functions satisfy the coherence conditions necessary
for a Kripke model.

To have a Kripke model, $f$ must also respect =$_{M}$ and
$\in_{M}$, meaning that $f$ must be an =$_{M}$- and
$\in_{M}$-homomorphism (i.e. $\sigma =_{M} \tau \; \rightarrow \;
f(\sigma$) =$_{M} f(\tau$), and similarly for $\in_{M}$). In order
for these to be true, we need an additional restriction on the
model. By way of motivation, one requirement is, intuitively
speaking, that the sets $\sigma$ can't shrink as we go to later
nodes. That is, once $\sigma_{j}$ gets into $\sigma$ at some node,
it can't be thrown out at a later node. $\sigma_{j}$ gets into
$\sigma$ because $r$ is compatible with $(p_j, I_{j})$ (where
$\langle \sigma_j, (p_j, I_j) \rangle \in \sigma$). So we need to
guarantee that if $r$ and $(p, I)$ are compatible and $r'$ is
associated to any extending node then $r'$ and $(p, I)$ are
compatible for any condition $(p, I)$. This holds exactly when
$r'$ extends $r$ and all of the entries in $r'\backslash r$ are
infinitesimally close to lim($r$). This happens, for instance,
when $r' = f(r)$. Other such examples would be $f(r)$ truncated at
some non-standard place and arbitrarily extended by any Cauchy
sequence through the reals with standard part lim($r$); in fact,
all such $r'$ have that form. We henceforth take this as an
additional condition on the construction: once $r$ is associated
to a node, then for any $r'$ associated to an extending node,
rng($r'\backslash r$) must consist only of rationals infinitely
close to lim($r$).

\begin{lemma}
f is an $=_{M}$ and $\in_{M}$-homomorphism.
\end{lemma}
\begin{proof}
If $\sigma =_{M} \tau$ then let $(p, I)$ compatible with $r$
witness as much. At any later node, $(p, I) = f((p, I)) = (f(p),
f(I)) \Vdash f(\sigma) =_{M} f(\tau)$. Also, the associated real
$r'$ would still be compatible with $(p, I)$. So the same $(p, I)$
would witness $f(\sigma) =_{M} f(\tau)$ at that node. Similarly
for $\in_{M}$.
\end{proof}

We can now conclude that we have a Kripke model.

\begin{lemma} \label{equalitylemma}
This Kripke model satisfies the equality axioms:
\begin{enumerate}
\item $\forall x \; x=x$

\item $\forall x, y \; x=y \rightarrow y=x$

\item $\forall x, y, z \; x=y \wedge y=z \rightarrow x=z$

\item $\forall x, y, z \; x=y \wedge x \in z \rightarrow y \in z$

\item $\forall x, y, z \; x=y \wedge z \in x \rightarrow z \in y.$

\end{enumerate}

\end{lemma}
\begin{proof}
1: It is easy to show with a simultaneous induction that, for all
$(p, I)$ and $\sigma$, $(p, I)$ $\Vdash \sigma =_{M} \sigma$, and,
for all $\langle \sigma_{j}, (p_j, I_j) \rangle \in \sigma$
compatible with $(p, I), (p \cup p_j, I \cap I_{j}) \Vdash
\sigma_{i} \in_{M} \sigma$.

2: Trivial because the definition of $(p, I)$ $\Vdash \sigma =_{M}
\tau$ is itself symmetric.

3: For this and the subsequent parts, we need some lemmas.
\begin{lemma} If (p', I') $\leq$ (p, I) $\Vdash \sigma =_{M}
\tau$ then (p', I') $\Vdash \sigma =_{M} \tau$, and similarly for $\in_{M}$.
\end{lemma}
\begin{proof} By induction on $\sigma$ and $\tau$.
\end{proof}

\begin{lemma} If (p, I) $\Vdash \rho =_{M} \sigma$ and (p, I) $\Vdash
\sigma =_{M} \tau$ then (p, I) $\Vdash \rho =_{M} \tau$.
\end{lemma}
\begin{proof} Again, by induction on terms.
\end{proof}

Returning to proving property 3, the hypothesis is that for some
$(p, I)$ and $(q, J)$ each compatible with $r$, $(p, I) \Vdash
\rho =_{M} \sigma$ and $(q, J) \Vdash \sigma =_{M} \tau$. By the
first lemma, $(p \cup q, I \cap J) \Vdash \rho =_{M} \sigma,
\sigma =_{M} \tau$, and so by the second, $(p \cup q, I \cap J)
\Vdash \rho =_{M} \tau$. Also, $(p \cup q, I \cap J)$ is
compatible with $r$.

4: Let $(p, I) \Vdash \rho =_{M} \sigma$ and $(q, J) \Vdash \rho
\in_{M} \tau$. We will show that $(p \cup q, I \cap J) \Vdash
\sigma \in_{M} \tau$. Let {\it C} be a cover of $(q, J)$
witnessing $(q, J) \Vdash \rho \in_{M} \tau$. We will show that
$(p \cup q, I \cap J) \wedge {\it C} = (p, I) \wedge {\it C}$ is a
cover of $(p \cup q, I \cap J)$ witnessing $(p \cup q, I \cap J)
\Vdash \sigma \in_{M} \tau$. Let $(q_i, J_i) \in$ {\it C} and
$\langle \tau_k, p_k, I_k \rangle$ be the corresponding member of
$\tau$. By the first lemma, $(p \cup q_i, I \cap J_{i}) \Vdash
\rho =_{M} \sigma$, and so by the second, $(p \cup q_i, I \cap
JS_{i}) \Vdash \sigma =_{M} \tau_{k}$.

5: Similar, and left to the reader.
\end{proof}

With this lemma in hand, we can now mod out by =$_{M}$, so that the
symbol ``=" is interpreted as actual set-theoretic equality. We
will henceforth drop the subscript $_{M}$ from = and $\in$,
although we will not distinguish notationally between a term
$\sigma$ and the model element it represents, $\sigma$'s
equivalence class.

At this point, we need to finish specifying the model in detail.
What remains to be done is to associate a Cauchy sequence to each
node. At the bottom level, assign each Cauchy sequence from $M_0$
to exactly one node. Inductively, suppose we chose have the
sequence $r$ at a node with ground model $M_{n}$. There are
continuum-in-the-sense-of-$M_{n+1}$-many immediate successor
nodes. Associate each possible candidate $r'$ in $M_{n+1}$ with
exactly one such node. (As a reminder, that means each member of
rng($r'$$\backslash$$r$) is infinitely close to lim($r$).)

By way of notation, a node will be named by its associated
sequence. Hence ``$r$ $\models \phi$" means $\phi$ holds at the
node with sequence $r$.

Note that, at any node of level $n$, the choice of $r$'s from that
node on is definable in $M_{n}$. This means that the evaluation of
terms (at and beyond the given node) can be carried out over
$M_{n}$, and so the Kripke model (from the given node on) can be
defined over $M_{n}$, truth predicate and all.

\subsection{The Forcing Relation}
Which $(p, I)$'s count as true determines the interpretation of
all terms, and hence of truth in the end model. We need to get a
handle on this. As with forcing, we need a relation $(p, I) \Vdash
\phi$ which supports a truth lemma. Note that, by elementarity, it
doesn't matter in which classical model $M_n$ or at what node in
the ground Kripke model $\Vdash$ is being interpreted (as long as
the parameters are in the interpreting model, of course).

\begin{definition}
(p, I) $\Vdash \phi$ is defined inductively on $\phi$:

(p, I) $\Vdash \sigma =_{M} \tau$ iff for all
$\langle \sigma_{j}, (p_{j}, I_j) \rangle \in \sigma$ compatible with
(p, I) $(p \cup p_j, I \cap I_{j}) \Vdash \sigma_{j} \in_{M} \tau$ and vice
versa

(p, I) $\Vdash \sigma \in_{M} \tau$ iff there is a cover {\it C}
of (p, I) such that for all $(p_j, I_j) \in {\it C}$ there is a
$\langle \tau_{k}, (p_{k}, I_k) \rangle \in \tau$ such that $(p_j,
I_j) \leq (p_k, I_k)$ and $(p_j, I_j) \Vdash \sigma =_{M}
\tau_{k}$.

(p, I) $\Vdash \phi \wedge \psi$ iff (p, I) $\Vdash \phi$ and (p, I) $\Vdash
\psi$

(p, I) $\Vdash \phi \vee \psi$ iff there is a cover {\it C} of (p,
I) such that, for each (p$_j$, I$_j$) $\in$ {\it C}, (p$_j$, I$_j$)
$\Vdash \phi$ or (p$_j$, I$_j$) $\Vdash \psi$

(p, I) $\Vdash \phi \rightarrow \psi$ iff for all (q, J) $\leq$ (p, I) if
(q, J) $\Vdash \phi$ then (q, J) $\Vdash \psi$

(p, I) $\Vdash \exists x \; \phi(x)$ iff there is a cover {\it C} of (p,
I) such that, for each (p$_j$, I$_j$) $\in$ {\it C}, there is a $\sigma$ such that
(p$_j$, I$_j$) $\Vdash
\phi(\sigma)$

(p, I) $\Vdash \forall x \; \phi(x)$ iff for all  $\sigma$ (p, I) $\Vdash \phi(\sigma)$

\end{definition}

\begin{lemma} \label{helpful lemma}
\begin{enumerate}
\item If (q, J) $\leq$ (p, I) $\Vdash \phi$ then (q, J) $\Vdash \phi$.
\item If {\it C} covers (p, I), and (p$_j$, I$_{j}$) $\Vdash \phi$ for all
(p$_j$, I$_{j}$) $\in$ {\it C}, then (p, I) $\Vdash \phi$.
\item (p, I) $\Vdash \phi$ iff for all r compatible with (p, I) there is a
(q, J) compatible with r such that (q, J) $\Vdash \phi$.
\item Truth Lemma: For any node r, r $\models \phi$
iff (p, I) $\Vdash \phi$ for some (p, I) compatible with r.
\end{enumerate}
\end{lemma}

\begin{proof}
1. A trivial induction, using of course the earlier lemmas about $\leq$ and covers.

2. Easy induction. The one case to watch out for is $\rightarrow$,
where you need to invoke the previous part of this lemma.

3. Trivial, using 2.

4. By induction on $\phi$, in detail for a change.

In all cases, the right-to-left direction (``forced implies true") is pretty easy,
by induction. (Note that only the $\rightarrow$ case needs the
left-to-right direction in this induction.) Hence in the following
we show only left-to-right (``if true at a node then forced").

=: This is exactly the definition of =.

$\in$: This is exactly the definition of $\in$.

$\wedge$: If $r \models \phi \wedge \psi$, then $r \models \phi$
and $r \models \psi$. Inductively let $(p, I) \Vdash \phi$ and
$(q, J) \Vdash \psi$, where $(p, I)$ and $(q, J)$ are each
compatible with $r$. That means that $(p, I)$ and $(q, J)$ are
compatible with each other, and $(p \cup q, I \cap J)$ suffices.

$\vee$: If $r \models \phi \vee \psi$, then WLOG $r \models \phi$
. Inductively let $(p, I)$ $\Vdash \phi$, $(p, I)$ compatible with
$r$. \{$(p, I)$\} suffices.

$\rightarrow$: Suppose to the contrary $r$ $\models \phi
\rightarrow \psi$ but no $(p, I)$ compatible with $r$ forces such.
Work in the node $f(r)$. (Recall that $f$ is the universal symbol
for the various transition functions in sight. What we mean more
specifically is that if $r$ $\in M_n$, i.e. if $r$ is a node from
level $n$, then $f(r)$ is the image of $r$ in $M_{n+1}$, i.e. in
the Kripke structure on level $n+1$.) Let $(p, I)$ be compatible
with $f(r)$ and $p$ have non-standard (in the sense of $M_n$)
length (equivalently, $I$ has infinitesimal length). Since $(p,
I)$ $\not\Vdash \phi \rightarrow \psi$ there is a $(q, J)$ $\leq$
$(p, I)$ such that $(q, J)$ $\Vdash \phi$ but $(q, J)$ $\not\Vdash
\psi$.  By the previous part of this lemma, there is an $r'$
compatible with $(q, J)$ such that no condition compatible with
$r'$ forces $\psi$. At the node $r'$, by induction, $r'$
$\not\models \psi$, even though $r'$ $\models \phi$ (since $r'$ is
compatible with $(p, I)$ $\Vdash \phi$). This contradicts the
assumption on $r$ (i.e. that $r$ $\models \phi \rightarrow \psi$),
since $r'$ extends $r$ (as nodes).

$\exists$: If $r$ $\models \exists x \; \phi(x)$ then let $\sigma$
be such that $r$ $\models \phi(\sigma)$. Inductively there is a
$(p, I)$ compatible with $r$ such that $(p, I)$ $\Vdash
\phi(\sigma)$. \{$(p, I)$\} suffices.

$\forall$: Suppose to the contrary $r$ $\models \forall x \;
\phi(x)$ but no $(p, I)$ compatible with $r$ forces such. As with
$\rightarrow$, let $(p, I)$ non-standard be compatible with
$f(r)$. Since $(p, I)$ $\not\Vdash \forall x \; \phi(x)$ there is
a $\sigma$ such that $(p, I)$ $\not\Vdash \phi(\sigma)$. By the
previous part of this lemma, there is an $r'$ compatible with $(p,
I)$ such that, for all $(q, J)$ compatible with $r', (q, J)
\not\Vdash \phi(\sigma)$. By induction, that means that $r'
\not\models \phi(\sigma)$. This contradicts the assumption on $r$
(i.e. that $r$ $\models \forall x \; \phi(x)$), since $r'$ extends
$r$ (as nodes).

\end{proof}

\subsection{The Final Proofs}
Using $\Vdash$, we can now prove

\begin{theorem}
This Kripke model satisfies IZF$_{Ref}$.
\end{theorem}
\begin{proof}
Note that, as a Kripke model, the axioms of intuitionistic logic
are satisfied, by general theorems about Kripke models.

\begin{itemize}
\item Empty Set: The interpretation of the term $\emptyset$ will do.
\item Infinity: The canonical name for $\omega$ will do. (Recall
that the canonical name $\bar{x}$ of any set $x \in V$ is defined
inductively as $\{ \langle \bar{y}, (\emptyset, \mathbb{R})
\rangle \mid y \in x \}.)$
\item Pairing: Given $\sigma$ and $\tau$, $\{ \langle \sigma,
(\emptyset, {\mathbb R}) \rangle , \langle \tau, (\emptyset,
{\mathbb R}) \rangle \}$ will do.
\item Union: Given $\sigma$, $\{ \langle \tau, J \cap J_i \rangle
\mid$ for some $\sigma_i, \; \langle \tau, J \rangle \in \sigma_i$
and $\langle \sigma_i, J_i \rangle \in \sigma \}$ will do.
\item Extensionality: We need to show that $\forall x \; \forall y
\; [\forall z \; (z \in x \leftrightarrow z \in y) \rightarrow x =
y]$. So let $\sigma$ and $\tau$ be any terms at a node $r$ such
that $r \models ``\forall z \; (z \in \sigma \leftrightarrow z \in
\tau)"$. We must show that $r \models ``\sigma = \tau"$. By the
Truth Lemma, let $r \in J \Vdash ``\forall z \; (z \in \sigma
\leftrightarrow z \in \tau)"$; i.e. for all $r' \in J, \rho$ there
is a $J'$ containing $r'$ such that $J \cap J' \Vdash \rho \in
\sigma \leftrightarrow \rho \in \tau$. We claim that $J \Vdash
``\sigma = \tau"$, which again by the Truth Lemma suffices. To
this end, let $\langle \sigma_i, J_i \rangle$ be in $\sigma$; we
need to show that $J \cap J_i \Vdash \sigma_i \in \tau$. Let $r'$
be an arbitrary member of $J \cap J_i$ and $\rho$ be $\sigma_i$.
By the choice of $J$, let $J'$ containing $r'$ be such that $J
\cap J' \Vdash \sigma_i \in \sigma \leftrightarrow \sigma_i \in
\tau$; in particular, $J \cap J' \Vdash \sigma_i \in \sigma
\rightarrow \sigma_i \in \tau$. It has already been observed in
\ref{equalitylemma}, part 1, that $J \cap J' \cap J_i \Vdash
\sigma_i \in \sigma$, so $J \cap J' \cap J_i \Vdash \sigma_i \in
\tau$. By going through each $r'$ in $J \cap J_i$ and using
\ref{helpful lemma}, part 3, we can conclude that $J \cap J_i
\Vdash \sigma_i \in \tau$, as desired. The other direction
($``\tau \subseteq \sigma"$) is analogous.
\item Set Induction (Schema): Suppose $r$ $\models ``\forall x \;
((\forall y \in x \; \phi(y)) \rightarrow \phi(x))"$; by the Truth
Lemma, let $J$ containing $r$ force as much. We must show $r$
$\models ``\forall x \; \phi(x)"$. Suppose not. Using the
definition of satisfaction in Kripke models, there is an $r'$
extending (i.e. infinitesimally close to) $r$(hence in $J$) and a
$\sigma$ such that $r'$ $\not \models \phi(\sigma)$. By
elementarity, there is such an $r'$ in $M_n$, where $n$ is the
level of $r$. Let $\sigma$ be such a term of minimal $V$-rank
among all $r'$s in $J$. Fix such an $r'$. By the Truth Lemma (and
the choice of $J$), $r'$ $\models ``(\forall y \in \sigma \;
\phi(y)) \rightarrow \phi(\sigma)"$. We claim that $r'$ $\models
``\forall y \in \sigma \; \phi(y)"$. If not, then for some $r''$
extending $r'$ (hence in $J$) and $\tau, r'' \models \tau \in
\sigma$ and $r'' \not \models \phi(\tau)$. Unraveling the
interpretation of $\in$, this choice of $\tau$ can be substituted
by a term $\tau$ of lower $V$-rank than $\sigma$. By elementarity,
such a $\tau$ would exist in $M_n$, in violation of the choice of
$\sigma$, which proves the claim. Hence $r'$ $\models
\phi(\sigma)$, again violating the choice of $\sigma$. This
contradiction shows that $r \models ``\forall x \; \phi(x)"$.
\item Separation (Schema): Let $\phi(x)$ be a formula and $\sigma$
a term. Then $\{ \langle \sigma_i, J \cap J_i \rangle \mid \langle
\sigma_i, J_i \rangle \in \sigma$ and $J \Vdash \phi(\sigma_i) \}$
will do.
\item Power Set: A term $\hat{\sigma}$ is a canonical subset of
$\sigma$ if for all $\langle \sigma_i, \hat{J_i} \rangle \in
\hat{\sigma}$ there is a $J_i \supseteq \hat{J_i}$ such that
$\langle \sigma_i, {J_i} \rangle \in \sigma$. $\{ \langle
\hat{\sigma}, (\emptyset, {\mathbb R}) \rangle \mid \hat{\sigma}$
is a canonical subset of $\sigma \}$ is a set (in $M_n$), and will
do.
\item Reflection (Schema): Recall that the statement of Reflection
is that for every formula $\phi(x)$ (with free variable $x$ and
unmentioned parameters) and set $z$ there is a transitive set $Z$
containing $z$ such that $Z$ reflects the truth of $\phi(x)$ in
$V$ for all $x \in Z$. So to this end, let $\phi(x)$ be a formula
and $\sigma$ be a set at a node $r$ of level $n$ (in the tree
which is this Kripke model's partial order). Let $k$ be such that
the truth of $\phi(x)$ at node $r$ and beyond is $\Sigma_{k}$
definable in $M_n$. In $M_{n}$, let $X$ be a set containing
$\sigma$, $r$, and $\phi$'s parameters such that $X \prec_{k}
M_{n}$. Let $\tau$ be $\{ \langle \rho, (\emptyset, {\mathbb R})
\rangle \mid \rho \in X$ is a term\}. $\tau$ will do.
\end{itemize}
\end{proof}

We are interested in the canonical term $\{ \langle {\bar p}, (p,
I) \rangle \mid $ $p$ is a finite function from ${\mathbb N}$ to
${\mathbb Q}$ and $I$ is a non-empty, open interval from the reals
with rational endpoints$\}$, where ${\bar p}$ is the canonical
name for $p$. We will call this term $Z$. Note that at node $r$
$Z$ gets interpreted as $r$.

\begin{proposition}
For all nodes r, r $\models$ ``Z is a Cauchy sequence".
\end{proposition}
\begin{proof}
To see that $\bot \models$ ``$Z$ is total", suppose $r$ $\models$
``$N$ is an integer". Then ($\langle N, r(N) \rangle, {\mathbb
R})$ is compatible with $r$ and forces ``$Z(N) = r(N)$". That $Z$
is a function is similarly easy.

As for $Z$ being Cauchy, again let $r$ $\models$ ``$N$ is an
integer". Since $r$ is Cauchy, there is an integer $M$ such that
beyond $M$ $r$ stays within an interval $I$ of size $1/(2N)$. Of
course, future nodes might be indexed by Cauchy sequences $s$
extending $r$ that go outside of $I$, but only by an infinitesimal
amount. So let $J$ extend $I$ on either side and still have length
less that $1/N$. Then ($r \upharpoonright M, J$) is compatible
with $r$, and forces that $Z$ beyond $M$ stay in $J$, making $Z$
``Cauchy for $1/N$", to coin a phrase.
\end{proof}

In order to complete the theorem, we need only prove the following
\begin{proposition}
For all nodes r, r $\models$ ``Z has no modulus of convergence."
\end{proposition}
\begin{proof}
Suppose $r \models ``f$ is a modulus of convergence for $Z$." Let
$(p, I)$ compatible with $r$ force as much. WLOG $I$ is a finite
interval. Let $n$ be such that $1/n$ is less than the length of
$I$, and let $\epsilon$ be (length($I) - 1/n)/2$.

Let $(q, J) \leq (p, I)$ force a value $m$ for $f(n)$. WLOG
length$(q) > m$, as $q$ could be so extended. If $J=I$, then $(q,
J)$ could be extended simply by extending $q$ with two values a
distance greater than $1/n$ apart, thereby forcing $f$ not to be a
modulus of convergence. So $J \subset I$. That means either inf
$J$ $>$ inf $I$ or sup $J <$ sup $I$. WLOG assume the latter. Let
mid $J$ be the midpoint of $J, \; q_0$ be $q$ extended by mid $J$,
and $J_0$ be (mid $J$, sup $J$). Then $(q_0, J_0) \leq (q, J)$,
and therefore $(q_0, J_0) \Vdash f(n) = m$.

What $(q_1, J_1$) is depends:

CASE I: There is an open set $K$ containing sup($J_0$) such that
$(q_0, K) \Vdash f(n) = m$. Then let $j_{max}$ be the sup of the
right-hand endpoints (i.e. sups) of all such $K$'s. Let $q_1$ be
$q_0$ extended by sup $J_0$.

\underline{Claim}: $(q_1, (\sup J_0, j_{max})) \Vdash f(n) = m.$

\begin{proof}
Let $r$ be any Cauchy sequence compatible with $(q_1$, (sup J$_0,
j_{max}$)). Since lim $r < j_{max}$, $r$ (that is, rng($r
\backslash q_0$)) is actually bounded below $j_{max}$. By the
definition of $j_{max}$, there is an open $K$ containing sup $J_0$
such that $r$ is bounded by sup $K$. As $r$ is bounded below by
sup $J_0$, $r$ (again, rng($r \backslash q_0$)) is contained
within $K$. As ($q_0, K) \Vdash f(n) = m, \; r \models f(n) = m$.
\end{proof}
Of course, $J_1$ will be (sup $J_0, j_{max}$).

CASE II: Not Case I. Then extend by the midpoint again. That is,
$q_1$ is $q_0$ extended by mid $J_0$, and $J_1$ is (mid $J_0$, sup
$J_0$). Also in this case, $(q_1, J_1) \Vdash f(n) = m$.

Clearly we would like to continue this construction. The only
thing that might be a problem is if the right-hand endpoint of
some $J_k$ equals (or goes beyond!) sup $I$, as we need to stay
beneath $(p, I)$. In fact, as soon as sup($J_k) >$ sup$(I) -
\epsilon$ (if ever), extend $q_k$ by something within $\epsilon$
of sup $I$, and continue the construction with left and right
reversed. That is, instead of going right, we now go left. This is
called ``turning around".

What happens next depends.

CASE A: We turned around, and after finitely many more steps, some
$J_k$ has its inf within inf($I) + \epsilon$. Then extend $q_k$ by
something within $\epsilon$ of inf $I$. This explicitly blows $f$
being a modulus of convergence for $Z$.

CASE B: Not Case A. So past a certain point (either the stage at
which we turned around, or, if none, from the beginning) we're
marching monotonically toward one of $I$'s endpoints, but will
always stay at least $\epsilon$ away. WLOG suppose we didn't turn
around. Then the construction will continue for infinitely many
stages. The $q_k$'s so produced will in the limit be a (monotonic
and bounded, hence) Cauchy sequence $r$. Furthermore, lim $r$ is
the limit of the sup($J_k$)'s. Finally, $r$ is compatible with
$(p, I)$. Hence $r \models ``f$ is total", and so $r \models f(n)
= m'$, for some $m'$. Let some condition compatible with $r$ force
as much. This condition will have the form $(q', K)$, where lim $r
\in K$. Let $k$ be such that sup($J_k) \in K$, and $q_k \supseteq
q'$. Consider $(q_k, J_k$). Note that $(q_k, J_k \cap K)$ extends
both $(q_k, J_k$) and $(q', K)$, hence forces both $f(n) = m$ and
$f(n) = m'$, which means that $m=m'$.

Therefore, at this stage in the construction, we are in Case I. By
the construction, $J_{k+1} = (\sup J_k, j_{max}$), where $j_{max}
\geq$ sup $K >$ lim $r$ = lim$_k (\sup(J_k)) \geq \sup(J_{k+1}) =
j_{max}$, a contradiction.
\end{proof}

\section{Not every Cauchy sequence of Cauchy sequences converges}

\begin{theorem} \label{thm2}
IZF$_{Ref}$ does not prove that every Cauchy sequence of Cauchy
sequences converges to a Cauchy sequence.
\end{theorem}

The statement of the theorem itself needs some elaboration. The
distance $d(x_{0n}, x_{1n})$ between two Cauchy sequences $x_{0n}$
and $x_{1n}$ is the sequence $\mid x_{0n} - x_{1n} \mid$. $x_{0n}
< x_{1n}$ if there are $m, N \in {\mathbb N}$ such that for all $k
> N \; x_{0k} + 1/m < x_{1k}$. A rational number $r$ can be
identified with the constant Cauchy sequence $x_{n} = r. \; x_n$ =
0 if $\forall m \; \exists N \; \forall k > N \; \mid x_k \mid <
1/m$. $x_{0n}$ and $x_{1n}$ are equal (as reals, equivalent as
Cauchy sequences if you will) if $d(x_{0n}, x_{1n})$ = 0. With
these definitions in place, we can talk about Cauchy sequences of
Cauchy sequences, and limits of such. The theorem is then that it
is consistent with IZF$_{Ref}$ to have a convergent sequences of
Cauchy sequences with no limit.

Note that we are not talking about reals! A real number would
be an equivalence class of Cauchy sequences (omitting, for the
moment, considerations of moduli of convergence). It would be weaker
to claim that the sequence of reals represented by the constructed
sequence of sequences has no limit. After all, given a sequence of
reals, it's not clear
that there is a way to choose a Cauchy sequence from each real. We
are claiming here that even if your task is made easier by being
handed a Cauchy sequence from each real, it may still not be
possible to get a ``diagonalizing", i.e. limit, Cauchy sequence.

\subsection{The Topological Space and Model}

Let $T$ be the space of Cauchy sequences of Cauchy sequences. By
way of notation, if $X$ is a member of $T$, then $X_j$ will be the
$j^{th}$ Cauchy sequence in $X$; as a Cauchy sequence of
rationals, $X_j$ will have values $X_j(0), X_j$(1), etc. Still
notationally, if $X_n \in T$, then the $j^{th}$ sequence in $X_n$
is $X_{nj}$. In the classical meta-universe, the Cauchy sequence
$X_j$ has a limit, lim($X_j$); in addition, the sequence $X$ has a
limit, which will be written as lim($X$).

A basic open set $p$ is given by a finite sequence $\langle (p_j,
I_j) \mid j<n_p \rangle$ of basic open sets from the space of the
previous theorem (i.e. $p_j$ is a finite sequence of rationals and
$I_j$ is an open interval), plus an open interval $I_p$. $X \in p$
if $X_j \in (p_j, I_j)$ for each $j < n_p$, if lim($X_j$) $\in
I_p$ for each $j \geq n_p$, and lim($X$) $\in I$. Note that $q
\subseteq p$ ($q$ \underline{extends} $p$) if $n_q \geq n_p, (q_j,
K_j) \subseteq (p_j, I_j)$ for $j < n_p, K_j \subseteq I_p$ for $j
\geq n_p$, and $I_q \subseteq I_p$.

$p$ and $q$ are compatible (where WLOG $n_p \leq n_q$) if, for $j
< n_p \; (p_j, I_j$) and ($q_j, K_j$) are compatible, for $n_p
\leq j < n_q \; K_j \cap I_p \not = \emptyset$, and $I_q \cap I_p
\not = \emptyset$. In this case, $p \cap q$ is not the basic open
set you'd think it is, but rather a union of such. The problem is
that for $n_p \leq j < n_q$ it would be too much to take the
$j^{th}$ component to be $(q_j, K_j \cap I_p$), because that would
leave out all extensions of $q_j$ with entries from $K_j
\backslash I_p$ before they finally settle down to $K_j \cap I_p$.
So $p \cap q$ will instead be covered by basic open sets in which
the $j^{th}$ component will be ($r_j, K_j \cap I_p$), where ($r_j,
K_j) \subseteq (q_j, K_j$). (So the given basic open sets form not
a basis for the topology, but rather a sub-basis.)

As always, the sets in the induced Heyting-valued model $M$ are of
the form $\{\langle \sigma_{k}, p_{k} \rangle \mid k \in K \}$,
where $K$ is some index set, each $\sigma_{k}$ is a set
inductively, and each $p_{k}$ is an open set. Note that all sets
from the ground model have canonical names, by choosing each
$p_{k}$ to be $T$ (i.e. $n_p$ = 0 and $I_p = {\mathbb R}$),
hereditarily. $M$ satisfies IZF$_{Ref}$.

\subsection{The Extensions $\leq_j$ and $\leq_\infty$}
In the final proof, we will need the following notions.

\begin{definition}
\underline{{\it j}-extension} $\leq_j$: q $\leq_j$ p for some $j <
n_p$ if q and p
satisfy all of the clauses of q extending p except possibly for
the condition on the $j^{th}$ component: (q$_j$, K$_j$) need not
be a subset of $(p_j, I_j)$, although we will still insist that
(q$_j$, I$_j$) be a subset of $(p_j, I_j)$.
\end{definition}
More concretely, $q_j$ comes from $p_j$ by extending with elements
from $I_j$; it's just that we're no longer promising to keep to
$I_j$ in the future. Notice that $\leq_j$ is not transitive; the
transitive closure of $\leq_j$ will be notated as $\leq_j^*$.

\begin{definition}
\underline{$\infty$-extension} $\leq_\infty$: q
$\leq_\infty$ p if q and p satisfy all of the clauses of q
extending p except possibly for the last, meaning that I$_q$ need
not be a subset of I$_p$.
\end{definition}
$\leq_\infty^*$ is the transitive
closure of $\leq_\infty$.

\begin{lemma} Suppose q $\subseteq$ p, q $\subseteq \|$f(n)=m$\|$
for some particular m and n, and $j < n_p$. Then for all x $\in$ I$_j$ there is an
r $\subseteq$ p, r $\leq_j^*$ q such that r $\subseteq \|$f(n)=m$\|$
and x $\in$ L$_j$, where (r$_j$, L$_j$) is r's $j^{th}$ component.
\end{lemma}
\begin{proof}
If $x \in K_j$, then we are done: let $r$ be $q$. So assume WLOG
that $x \geq \sup(K_j$). The inspiration for this construction is
the construction of the previous theorem. The main difference is
that not only do we have ($q_j, K_j$) to contend with, we also
have all of $q$'s other components around. Hence the notion of a
$j$-extension: we will do the last theorem's construction on the
$j^{th}$ coordinates, and leave all the others alone.

First off, we would like to show that $q$ has a $j$-extension $q'
\subseteq p$ also forcing $f(n)=m$ such that sup($K_j) \in K'_j$.
Toward this end, let $X \in T$ be a member of (the open set
determined by) $q$ except that lim($X_j$) = sup($K_j$). $X$ is in
$p$, so there is some $r \subseteq p$ such that $X \in r$ and $r$
forces a value for $f(n)$, say $m'. \; q$ and $r$ are compatible
though: apart from the $j^{th}$ component, $X$ is in both, and the
only thing happening in the $j^{th}$ component is that, in $r$,
sup($K_j$) $\in L_j$, meaning that $K_j$ and $L_j$ overlap. So any
common extension of both $q$ and $r$ would have to force $f(n)=m$
and $f(n)=m'$; since $p$ already forces that $f$ is a function,
$m=m'$. Using $r$, it is easy to construct the desired $q'$: take
the $j^{th}$ component from $r$, and let each other component be
the intersection of the corresponding components from $r$ and $q$.

If there is such a $q'$ such that $x \in K'_j$, then we are done.
Else we would like to mimic the last theorem's construction by
having in our next condition the interval part of the $j^{th}$
component be (sup($K_j$), $j_{max}$) (for a suitably defined
$j_{max}$). The problem is, $q$ has all these other components
around. For any real $y < j_{max}$ we could find a $j$-extension
of $q$ with (sup($K_j$), $y$) in the $j^{th}$ component, but not
necessarily for $y = j_{max}$ itself.

To this end, consider all such $q'$ as above. Each $q'$ can be
extended (to say $q''$) by restricting the interval in the
$j^{th}$ component to (sup($K_j$), sup($K'_j$)). Let $q_1$ be such
a $q''$ where that interval is at least half as big as possible
(i.e. among all such $q''$, where of course sup($K'_j$) has to be
bounded by sup($I_j$)).

Continue this construction so that $q_n$ is defined from $q_{n-1}$
just as $q_1$ was defined from q. WLOG dovetail this construction
with extending all other components so that after infinitely many
steps we would have produced an $X \in T$. (This remark needs a
word of justification about the $j^{th}$ components. By the
definition of $j$-extension alone, it is not clear that a sequence
of $j$-extending conditions $q_0 \geq_j q_1 \geq_j$ ... converges
to a point in $T$. In our case, though, by the construction
itself, the various $K_{nj}$'s are monotonically increasing and
bounded, hence the $X_j$ so determined is Cauchy.)

If at some finite stage we have covered $x$, then we are done. If
not, then sup($X_j$) = sup$_n$(sup($K_{nj}$)) $\leq x \in I_j$, so
that $X \in p$. So there is some $r \subseteq p$ with $X \in r$
such that $r$ forces a value for $f(n)$, say $m'$. Let $\epsilon$
be sup($L_j) - \sup(X_j$). Eventually in the construction,
$K_{nj}$ will be contained within $\epsilon$ of sup($X_j$). With
$r$ as the witness, at the next stage $K_{(n+1)j}$ would go beyond
sup($X_j$), which is a contradiction. Hence this case is not
possible, and at some finite stage we must have covered $x$, as
desired.

\end{proof}

We have a similar lemma for $\infty$-extensions.

\begin{lemma} Suppose q $\subseteq$ p and q $\subseteq \|$f(n)=m$\|$
for some particular m and n. Then for all x $\in$ I$_p$ there is an
r $\subseteq$ p, r $\leq_\infty^*$ q such that r $\subseteq \|$f(n)=m$\|$
and x $\in$ I$_r$.
\end{lemma}
\begin{proof}
Similar to the above.
\end{proof}

Observe that the same arguments work for preserving finitely many
values of $f$ simultaneously.

\subsection{The Final Proof}
We are interested in the canonical term $\{ \langle {\bar p_j}, p
\rangle \mid p$ is an open set$\}$, where ${\bar p_j}$ is the
canonical name for the sequence $\langle p_j \mid j<n_p \rangle$
from $p$. We will call this term $Z$. It should be clear that $T =
\| Z$ is a Cauchy sequence of Cauchy sequences$\|$. Hence we need
only prove
\begin{proposition} T = $\|$Z does not have a limit$\|$.
\end{proposition}

\begin{proof}
Suppose $p \subseteq \| f$ is a Cauchy sequence$\|$. It suffices
to show that for some $q \subseteq p, \; q \subseteq \| f \not =$
lim$(Z)\|$.

If $p$ can ever be extended to force infinitely many values for
$f$ simultaneously, then do so, and further extend (it suffices
here to extend merely the last component) to force $Z$ away from
$f$'s limit. This suffices for the theorem.

If this is not possible, then the construction will be to build
one or two sequences of open sets, $p_k$ and possibly $r_k$,
indexed by natural numbers $k$. It is to be understood even though
not again mentioned that the construction below is to be
dovetailed with a countable sequence of moves designed to produce
a single member of $T$ in the end (i.e. each individual component
must shrink to a real as in the previous theorem, the $n_{p_k}$'s
must be unbounded as $k$ goes through ${\mathbb N}$, and the last
components $I_{p_k}$ must shrink to something of length 0).

First, let $p_0$ be built by extending $p$ by cutting $I_p$ to its
bottom third, and let $L$ be some point in $I_p$'s top half. If
$p_0$ can be extended (to $p_1$) so that $f$ is forced to have an
additional value (that is, beyond what has already been forced) in
$I_p$'s top half, then do so. Else proceed as follows. First
extend $p_0$ to force an additional value for $f$, necessarily in
$I_p$'s bottom half. Then by the second lemma above,
$\infty$-extend that latter condition, to $q$ say, preserving the
finitely many values of $f$ already determined, and getting $L$
into $I_q$. Typically $n_q > n_p$, so let $\bar{q}$ be such that
$n_{\bar{q}} = n_q$, if $j < n_p$ then $\bar{q}$'s $j^{th}$
component is the same as $p$'s, if $n_p \leq j < n_q$ then
$\bar{q}$'s $j^{th}$ component is ($\emptyset, I_p$), and
$I_{\bar{q}} = I_p$. Note that $q \subseteq \bar{q} \subseteq p$,
so we can apply the first lemma above to $q$ and $\bar{q}$.
Starting from $q$, iteratively on $j$ from $n_p$ up to $n_q$,
$j$-extend to get $L$ into the interval part of the $j^{th}$
component, while preserving the finitely many values of $f$
already determined. Call the last condition so obtained $r_0$.
Finally, $\infty$-extend $r_0$ to get the last component to be a
subset of $I_{p_0}$, while still preserving $f$ of course. Let
this latter condition be $p_1$.

Stages $k > 0$ will be similar. To start, if possible, extend
$p_k$ to force an additional value for $f$ in $I_p$'s top half.
Call this condition $p_{k+1}$.

If that is not possible, first extend $p_k$ to force a new value
for $f$, necessarily in $I_p$'s bottom half. Then $\infty$-extend
(to $q$ say) to get $L$ into the last component $I_q$. After that,
$j$-extend for each $j$ from $n_{r_i}$ to $n_q$ to get $L$ in
those components, where $i$ is the greatest integer less than $k$
such that $r_i$ is defined. (It bears mentioning that $r_h$ is
defined if and only if at stage $h$ we are in this case.) If need
be, shrink those components to be subsets of $I_{r_i}$, for the
purpose of getting $r_k \subseteq r_i$ (once we define $r_k$).
That last condition will be $r_k$. Next, $\infty$-extend $r_k$ to
get the last component to be a subset $I_{p_k}$. This final
condition is $p_{k+1}$.

This completes the construction.

If the second option happens only finitely often, let $k$ be
greater than the last stage where it happens. Then not only does
$p_k$ force lim($Z$) to be in $I_p$'s bottom third, as all $p_i$'s
do actually, but also $p_k$ is respected in the rest of the
construction: for $i > k$, $p_i \subseteq p_k$. Let $l \geq k$ be
such that $6/l <$ length($I_p$) (i.e. the distance between $I_p$'s
top half and bottom third is greater than $1/l$). Recall that $p
\subseteq \| f$ is a Cauchy sequence$\|$; that is, $p \subseteq \|
\forall \epsilon > 0 \; \exists N \; \forall m, n \geq N \; |f(m)
- f(n)| < \epsilon\|$. Since $1/l >$ 0, $p \subseteq \| \exists N
\; \forall m, n \geq N \; |f(m) - f(n)| < 1/l\|$. That means there
is a cover $C$ of $p$ such that each $q \in C$ forces a particular
value for $N$. Let $S$ be $\bigcap_{j \geq k} p_j$, and let $q \in
C$ contain $S$. Similarly, let ${\hat q}$ containing $S$ force a
value for $f(N)$. $q \wedge {\hat q} \wedge p_k$ is non-empty
because it contains $S$, and $q \wedge {\hat q} \wedge p_k$ forces
by the construction that $f(N)$ is in $I_p$'s top half, by the
choice of $q$ that lim($f$) is away from $I_p$'s bottom third, and
by choice of $k$ that lim($Z$) is in $I_p$'s bottom third. In
short, $q \wedge {\hat q} \wedge p_k \subseteq \| f \not =$
lim($Z)\|$.

Otherwise the second option happens infinitely often. Then we have
an infinite descending sequence of open sets $r_k$, and a similar
argument works. Let $S$ be $\bigcap_j r_j$, where the intersection
is taken only over those $j$'s for which $r_j$ is defined. Let $k$
be such that $r_k \subseteq \|$lim($Z) - $midpoint($I_p) <
\epsilon \|$, for some fixed $\epsilon >$ 0. Let $q$ containing
$S$ be such that, for a fixed value of $N$, $q \subseteq \|
\forall m, n \geq N \; |f(m) - f(n)| < \epsilon\|$. Let ${\hat q}$
force a particular value for $f(N)$, necessarily in $I_p$'s bottom
half. Again, $q \wedge {\hat q} \wedge p_k \subseteq \| f \not =$
lim($Z)\|$.

\end{proof}

\section{The given Cauchy sequence has a modulus, but the limit doesn't}

\begin{theorem} \label{thm3}
IZF$_{Ref}$ does not prove that every Cauchy sequence with a modulus
of convergence of Cauchy sequences converges to a Cauchy sequence
with a modulus of convergence.
\end{theorem}

\begin{definition} c is a convergence function for a Cauchy sequence $\langle X_j
\mid j \in \bf{N} \rangle$ if c is a decreasing sequence of
positive rationals; for all n, if j, k
$\geq$ n then $\mid X_j - X_k \mid \leq c(n)$; and lim(c(n)) = 0.
\end{definition}
Notice that convergence functions and moduli of convergence are
easily convertible to each other: if $c$ is the former, then
$d(n)$ := the least $m$ such that $c(m) \leq 2^{-n}$ is the
latter; and if $d$ is the latter, then $c(n) := 2^{-m}$, where $m$
is the greatest integer such that max($m, d(m)) \leq n$, is the
former. Therefore the current construction will be of a Cauchy
sequence $\langle X_j \mid j \in {\bf N} \rangle$ with a
convergence function but no limit. Without loss of generality, the
convergence function in question can be taken to be $c(n) =
2^{-n}$.

Let the topological space $T$ be $\{ \langle X_j \mid j \in {\bf
N} \rangle \mid \langle X_j \mid j \in \bf{N} \rangle$ is a Cauchy
sequence of Cauchy sequences with convergence function $2^{-n}$
\}. As in the previous section, for $X \in T, \; X_j$ will be the
$j^{th}$ Cauchy sequence in $X$'s first component. The real number
represented by $X_j$, i.e. $X_j$'s limit, will be written as
lim($X_j$). In the classical meta-universe, the limit of the
sequence $\langle X_j \mid j \in \bf{N} \rangle$ will be written
as lim($X$).

$T$ is a subset of the space from the previous section, and the
topology of $T$ is to be the subspace topology. That is, a basic
open set $p$ is given again by a finite sequence $\langle (p_j,
I_j) \mid j<n_p \rangle$ and an open interval $I_p$. $X \in p$ if,
again, $X_j \in (p_j, I_j)$ for each $j<n_p$, lim($X_j$) $\in I_p$
for each $j\geq n_p$, and lim($X$) $\in I_p$. $p$ and $q$ are
compatible under the same conditions as before, and $p \cap q$ is
covered by basic open sets, just as in the last theorem; the
convergence function causes no extra trouble.

Note that $q \subseteq p \; (q$ \underline{extends} $p$) if all of
the same conditions from the last section hold: $n_q \geq n_p$,
($q_j, K_j) \subseteq (p_j, I_j)$ for $j < n_p, \; K_j \subseteq
I_p$ for $j \geq n_p$, and $I_q \subseteq I_p$.

In the following, we will need to deal with basic open sets in
canonical form. The issue is the following. Suppose, in $p$, $I_0$
= (0, 1) and $I_1$ = (0, 10). Then $X_1$ could certainly contain
elements from (0, 10). However, when it comes to taking limits,
$X_1$ has 2 as an upper bound, because of $I_0$ and the
convergence function $2^{-n}$, but this is not reflected in $I_1$.

\begin{definition} p is in canonical form if, for $j < k < n_p,
\mid$sup(I$_j$) - sup(I$_k)\mid \leq 2^{-j}$, and also
$\mid$sup(I$_j$) - sup(I$_p)\mid \leq 2^{-j}$.
\end{definition}

The value of canonical form is that, if for $j < n_p$ lim($X_j) =
\sup (I_j)$ and if lim($X$) = sup($I_p$), then, although $X \not
\in p$, $X$ could still be in $T$.

\begin{proposition} Every open set is covered by open sets in canonical
form.
\end{proposition}
\begin{proof}
Let $X \in p$ open. If, in $q \subseteq p$, $J_k$ is an interval
with midpoint lim($X_k$) and radius independent of $k$, and $I_q$
an interval with midpoint lim($X$) and the same radius, then $q$
will be canonical. We will construct such a $q$ containing $X$.

By way of choosing the appropriate radius, as well as $n_q$, let
$\delta$ be half the distance from lim($X$) to the closer of
$I_p$'s endpoints. Let $N \geq n_p$ be such that for all $k \geq
N$ lim($X_k$) is within $\delta$ of lim($X$). Let $r \leq \delta$
be such that for all $k < N$ (lim($X_k) - r$, lim($X_k) + r$)
$\subseteq I_k$. Let $n_q \geq N$ be such that for all $k \geq
n_q$ lim($X_k$) is within $r$ of lim($X$). For $k < n_q$ let $J_k$
be the neighborhood with center lim($X_k$) and radius $r$, and let
$q_k$ be an initial segment of $X_k$ long enough so that beyond it
$X_k$ stays within $J_k$. Let $I_q$ be the neighborhood with
center lim($X$) and radius $r$. This $q$ suffices.
\end{proof}

As always, the sets in the induced Heyting-valued model $M$ are of
the form $\{\langle \sigma_{k}, p_{k} \rangle \mid k \in K \}$,
where $K$ is some index set, each $\sigma_{k}$ is a set
inductively, and each $p_{k}$ is an open set. Note that all sets
from the ground model have canonical names, by choosing each
$p_{k}$ to be $T$ (i.e. $n_p$ = 0 and $I_p = {\mathbb R}$),
hereditarily. $M$ satisfies IZF$_{Ref}$.

We are interested in the canonical term $\{ \langle {\bar p_j}, p
\rangle \mid p$ is an open set$\}$, where ${\bar p_j}$ is the
canonical name for the sequences $\langle p_j \mid j<n_p \rangle$
from $p$. We will call this term $Z$. It should be clear that $T$
= $\|$Z is a Cauchy sequence of Cauchy sequences with convergence
function $2^{-n} \|$. Hence we need only prove
\begin{proposition} T = $\|$No Cauchy sequence equal to
lim(Z) has a modulus of convergence$\|$.
\end{proposition}

\begin{proof}
Suppose $p \subseteq \| f$ is a modulus of convergence for a
Cauchy sequence $g \|$, $p$ in canonical form. It suffices to show
that for some $q \subseteq p$, $q \subseteq \| g \not =$
lim($Z)\|$.

Let $\epsilon <$ (length $I_p$)/2. Let $q \subseteq p$ in
canonical form force ``$f(\epsilon) = N$"; WLOG $n_q > N$. We can
also assume (by extending again if necessary) that $q$ forces a
value for $g(N)$; WLOG $g(N) \leq$ midpoint($I_p$). Let $X \in p$
be on the boundary of $q$; that is, $X_k$ extends $q_k$ ($k <
n_q$), $X_k$ beyond length($q_k$) is a sequence through $J_k$ with
limit $\sup(J_k$), and $X$ beyond $n_q$ is a sequence through
$I_q$ with limit $\sup(I_q$) (more precisely, $\langle \lim(X_k)
\mid k \geq n_q \rangle$ is such a sequence).

(Technical aside: By the canonicity of $q$'s form, $X \in T$. But
why should $X$ be in $p$? This could fail only if $\sup(J_k) =
\sup(I_k$) or if $\sup(I_q) = \sup(I_p$). The latter case would
actually be good. The point of the current argument is to get a
condition $r$ (forcing the things $q$ forces) such that $I_r$
contains points greater than $g(N) + \epsilon$, which would fall
in our lap if $\sup(I_q) = \sup(I_p$). If $\sup(I_q) < \sup(I_p$)
and $\sup(J_k) = \sup(I_k$), then $X_k$ must be chosen so that
$\lim(X_k$) is slightly less than this sup. Could this interfere
with 2$^{-n}$ being a convergence function for $X$? No, by the
canonicity of $p$. If $l$ is another index such that $\sup(J_l) =
\sup(I_l$), then by letting $\lim(X_l$) be shy of this sup by the
same amount as for $k$ the convergence function 2$^{-n}$ is
respected (for these two indices). If $\sup(J_l) < \sup(I_l$),
then what to do depends on whether $\sup(J_k$) and $\sup(J_l$) are
strictly less than 2$^{-min(k,l)}$ apart or exactly that far
apart. In the former case, there's some wiggle room in the
$k^{th}$ slot for $\lim(X_k$) to be less than $\sup(J_k$). In the
latter, $\sup(J_l$) must be $\sup(J_k$) either increased or
decreased by 2$^{-min(k,l)}$. The first option is not possible, by
the canonicity of $p$, as $\sup(J_l) < \sup(I_l$). In the second
option, having $\lim(X_k$) be less than $\sup(J_k$) brings
$\lim(X_k$) and $\lim(X_l$) even closer together. Similar
considerations apply to comparing $\lim(X_k$) and $\lim_k(X_k) =
\sup(I_q$).)

Let $q_1$ in canonical form containing $X$ force values for
$f(\epsilon$) and $g(f(\epsilon$)). Since $q_1$ and $q$ are
compatible, they force the same such values. WLOG $q_1$ is such
that $\sup(I_{q_1}$) is big (that is, $\sup(I_{q_1}) - \sup(I_q$)
is at least half as big as possible). Continuing inductively,
define $q_{n+1}$ from $q_n$ as $q_1$ was defined from $q$.
Continue until $I_{q_n}$ contains points greater than $g(N) +
\epsilon$. This is guaranteed to happen, because, if not, the
infinite sequence $q_n$ will converge to a point $X$ in $p$. Some
neighborhood $r$ of $X$ forcing values for $f(\epsilon$) and
$g(f(\epsilon$)) will contain some $q_n$, witnessing that
$q_{n+1}$ would have been chosen with larger last component than
it was, as in the previous proofs.

Once the desired $q_n$ is reached, shrink $I_{q_n}$ to be strictly
above $g(N) + \epsilon$. Call this new condition $r$. $r$ forces
``$\lim(Z) > g(N) + \epsilon$", and $r$ also forces ``$g \leq g(N)
+ \epsilon$". So $r$ forces ``$g \not = \lim(Z$)", as desired.
\end{proof}

\section{The reals are not Cauchy complete}

\begin{theorem} \label{thm4}
IZF$_{Ref}$ does not prove that every Cauchy sequence of reals
has a limit.
\end{theorem}

As stated in the introduction, what we will actually prove will be
what seems to be the hardest version: there
is a Cauchy sequence, with its own modulus of convergence, of real
numbers, with no Cauchy sequence as a limit, even without a
modulus of convergence. Other versions are possible, such as
changing what does and doesn't have a modulus.
After all of the preceding proofs, and
after the following one, it should not be too hard for the reader to achieve any
desired tweaking of this version.

Let $T$ consist of all Cauchy sequences of Cauchy sequences, all
with a fixed convergence function of 2$^{-n}$. An open set $p$ is
given by a finite sequence $\langle (p_j, I_j) \mid j < n_p
\rangle$ as well as an interval $I_p$, with the usual meaning to
$X \in p$.

Recall from the previous section:
\begin{definition}
$p$ is in canonical form if, for each $j < k < n_p,$ $\mid
\sup(I_j) - \sup (I_k) \mid \leq 2^{-j}$. Also, $\mid \sup(I_j) -
\sup (I_p) \mid \leq 2^{-j}$.
\end{definition}

Also from the last section:
\begin{proposition}
Every open set is covered by sets in canonical form.
\end{proposition}

Henceforth when choosing open sets we will always assume they are
in canonical form.

\begin{definition}
$p$ and $q$ are similar, $p \sim q$, if $n_p = n_q, I_p = I_q, I_k
= J_k$, and length($p_k$) = length($q_k$). So $p$ and $q$ have the
same form, and can differ only and arbitrarily on the rationals
chosen for their components.

If moreover $p_k = q_k$ for each $k \in J$ then we say that $p$
and $q$ are $J$-similar, $p \sim_J q$.
\end{definition}

If $p \sim q$, this induces a homeomorphism on the topological
space $T$, and therefore on the term structure. (To put it
informally, wherever you see $p_k$, or an initial segment or
extension thereof, replace it (or the corresponding part) with
$q_k$, and vice versa. This applies equally well to members of
$T$, open sets, and (hereditarily) terms.)

\begin{definition} If $p, q$, and $r$ are open sets, $\sigma$ is a
term, and $q$ and $r$ are similar, then the image of $p$ under the
induced homeomorphism is notated by $p_{qr}$ and that of $\sigma$
by $\sigma_{qr}$.
\end{definition}

\begin{lemma}
p $\Vdash \phi(\vec{\sigma})$ iff p$_{qr} \Vdash \phi(\vec{\sigma}_{qr})$.
\end{lemma}

\begin{proof}
A straightforward induction.
\end{proof}

\begin{definition}
$\sigma$ has support $J$ if for all $p\sim_J q$ $\bot \Vdash
\sigma = \sigma_{pq}$. $\sigma$ has finite support if $\sigma$ has
support $J$ for some finite set $J$.
\end{definition}

The final model $M$ is the collection of all terms with
hereditarily finite support.

As always, let $Z$ be the canonical term. Note that $Z$ is not in
the symmetric submodel! However, each individual member of $Z$,
$Z_j$, is, with support $\{ j \}$. Also, so is $\langle [Z_j] \mid
j \in {\bf N} \rangle$, which we will call $[Z]$, with support
$\emptyset$. (Here, for $Y$ a Cauchy sequence, $[Y]$ is the
equivalence class of Cauchy sequences with the same limit as $Y$,
i.e. the real number of which $Y$ is a representative.) That's
because no finite change in $Z_j$ affects [$Z_j$]. (Notice that
even though each member of [$Z_j$] has support $\{ j \}$,
[$Z_j$]'s support is still empty.) It will ultimately be this
sequence [$Z$] that will interest us. But first:

\begin{proposition}
M $\models$ IZF$_{Ref}$.
\end{proposition}

\begin{proof}
As far as the author is aware,
symmetric submodels have been studied only in the context of
classical set theory, not intuitionistic, and, moreover, the only
topological models in the literature are full models, in which the
terms of any given model are all possible terms built on the space
in question, not submodels. Nonetheless, the same proof that the full model
satisfies IZF (easily, IZF$_{Ref}$) applies almost unchanged to
the case at hand. To keep the author honest without trying the
patience of the reader, only the toughest axiom, Separation, will
be sketched.

To this end, suppose the term $\sigma$ and formula $\phi$ have
(combined) support $J$ (where the support of a formula is the
support of its parameters, which are hidden in the notation used).
The obvious candidate for a term for the appropriate subset of
$\sigma$ is $\{ \langle \sigma_i, p \cap p_i \rangle \mid \langle
\sigma_i, p_i \rangle \in \sigma \wedge p \Vdash
\phi(\sigma_i)\}$, which will be called Sep$_{\sigma, \phi}$. We
will show that this term has support $J$.

To this end, let $q \sim_J r$. We need to show that $\bot \Vdash$
Sep$_{\sigma, \phi}$ = (Sep$_{\sigma, \phi})_{qr}$. In one
direction, any member of (Sep$_{\sigma, \phi})_{qr}$ is of the
form $\langle \sigma_i, p \cap p_i \rangle_{qr}$, where $\langle
\sigma_i, p \cap p_i \rangle \in$ Sep$_{\sigma, \phi}$, i.e.
$\langle \sigma_i, p_i \rangle \in \sigma$ and $p \Vdash
\phi(\sigma_i)$. We need to show that $(p \cap p_i)_{qr} \Vdash
(\sigma_i)_{qr} \in$ Sep$_{\sigma, \phi}$. Since $\bot \Vdash
\sigma = \sigma_{qr}$ and $(p_i)_{qr} \Vdash (\sigma_i)_{qr} \in
\sigma_{qr}, (p_i)_{qr} \Vdash (\sigma_i)_{qr} \in \sigma$. In
addition, by the lemma above, $p_{qr} \Vdash
\phi_{qr}((\sigma_i)_{qr})$ (where $\phi_{qr}$ is the result of
taking $\phi$ and applying the homeomorphism to its parameters).
Since $\phi$'s parameters have support $J$, $\bot \Vdash \phi_{qr}
= \phi$,  and $p_{qr} \Vdash \phi((\sigma_i)_{qr})$. Summarizing,
$(p \cap p_i)_{qr} \Vdash (\sigma_i)_{qr} \in \sigma \wedge
\phi((\sigma_i)_{qr})$, so $(p \cap p_i)_{qr} \Vdash
(\sigma_i)_{qr} \in $ Sep$_{\sigma, \phi}$, as was to be shown.

The other direction is similar.
\end{proof}

So there was no harm in taking the symmetric submodel. The benefit
of having done so is the following

\begin{lemma} Extension Lemma: Suppose q, r $\subseteq$ p, q
$\subseteq \| f(n) = m \|$ and for $j \in$ J (q$_j$, K$_j$) =
(r$_j$, L$_j$) (i.e. q and r agree on f's support). Then r has an
extension forcing f(n) = m.
\end{lemma}

\begin{proof}
Take a sequence of refinements of $q$ converging to a point $X$ on
$q$'s boundary, as follows. Consider $j < n_r, j \not \in J$. If
$K_j \cap L_j$ is non-empty, then just work within the latter set.
Else either $\sup(K_j) < \inf(L_j$), in which case let $\lim(X_j)
= \sup(K_j$), or $\inf(K_j) > \sup(L_j$), in which case let
$\lim(X_j) = \inf(K_j$). (In what follows, we will consider only
the first of those two cases.) Similarly for $I_q$ and $I_r$. As
usual, since $X \in p, \; X$ has a neighborhood forcing a value
for $f(n)$; since $X$ is on $q$'s boundary, any such neighborhood
has to force the same value for $f(n)$ that $q$ did. Let $q_1$ be
such a neighborhood where, for $j$ the smallest integer not in
$J$, $\sup((K_1)_j) - \sup(K_j$) is at least half as big as
possible.

To continue this construction, consider what would happen if
$\sup(K_1)_j < \inf(L_j$). We would like to take another point
$X$, this time on the boundary of $q_1$, with $\lim(X_j) =
\sup(K_1)_j$. The only possible obstruction is that ($q_1)_j$
might have entries far enough away from $\sup(K_1)_j$ so that the
constraint of the convergence function would prevent there being
such an $X$. In this case, change ($q_1)_j$ so that this is no
longer an obstruction. Since $j \not \in J$, the new condition is
$J$-similar to the old, and so will still force the same value for
$f(n)$.

Repeat this construction, making such that each of the finitely
many components $j < n_r$, $j \not \in J$ and the final component
get paid attention infinitely often (meaning $\sup((K_{n+1})_j) -
\sup((K_n)_j)$ is at least half as big as possible). This produces
a sequence $q_n$. Eventually $q_n$ will be compatible with $r$. If
not, let $X$ be the limit of the $q_n$'s. If $X \not \in r$ then,
for some component $j$, $\lim(X_j) < \inf(L_j$). $X$ has a
neighborhood, say $q_{\infty}$, forcing $f(n) = m$. At some large
enough stage at which $j$ gets paid attention, the existence of
$q_{\infty}$ would have made the $j^{th}$ component of the next
$q_n$ contain $\lim(X_j$), a contradiction.
\end{proof}

With the Extension Lemma in hand, the rest of the proof is easy.
It should be clear that [$Z$] has convergence function 2$^{-n}$.
So it remains only to show
\begin{proposition} $\|$ [Z] has no limit $\|$ = T.
\end{proposition}

\begin{proof}
Suppose $p \subseteq \| f$ is a Cauchy sequence $\|$. It suffices
to find a $q \subseteq p$ such that $q \subseteq \| f \not =
\lim([Z]) \|$.

By the Extension Lemma, all of $f$'s values are determined by
$f$'s finite support $J$. So $f$ cannot be a limit for [$Z$], as
any such limit has to be affected by infinitely many components.
\end{proof}

\section{Questions}
There is a variant of the questions considered nestled between the
individual Cauchy sequences of the big Cauchy sequence being
adorned with a modulus of convergence and not. It could be that
each such sequence has a modulus of convergence, but the sequence
is not paired with any modulus in the big sequence. Looked at
differently, perhaps the big sequence is one of Cauchy sequences
with moduli of convergence but not uniformly. Certainly this extra
information would not weaken any of the positive results. Would it
weaken any of the negative theorems though? Presumably not:
knowing that each of the individual sequences has a modulus doesn't
seem to help to build a limit sequence or a modulus for such, if
there's no way you can get your hands on them. Still, in the
course of trying to prove this some technical difficulties were
encountered, so the questions remain open.

The negative results here open up other hierarchies. Starting with
the rationals, one could consider equivalence classes of Cauchy
sequence with moduli of convergence. By the last theorem, that may
not be Cauchy complete. So equivalences classes can be taken of
sequences of those. This process can be continued, presumably into
the transfinite. Is there a useful structure theorem here? All of
this can be viewed as taking place inside of the Dedekind reals,
which are Cauchy complete. There is a smallest Cauchy complete set
of reals, namely the intersection of all such sets. As pointed out
to me by the referee, this could be a proper subset of the Dedekind
reals, since that is the case in the topological model of
\cite{FH}. Naturally enough, the same is also the case in the
Kripke model of \cite{LR}. Is there any interesting structure
between the Cauchy completion of the rationals and the Dedekind
reals? What about the corresponding questions for other notions of
reals, such as simply Cauchy sequences sans moduli?

As indicated in the introduction, the first two models, one
topological and the other Kripke, are essentially, even if not
substantially, different. What is the relation between the two?

In the presence of Countable Choice, all of the positive results
you could want here are easily provable (e.g. every Cauchy
sequence has a modulus of convergence, the reals are Cauchy
complete, etc.). Countable Choice itself,
though, is a stronger principle than necessary for this, since,
as pointed out to me by Fred
Richman, these positive results are true under classical logic,
but classical logic does not imply Countable Choice. Are there
extant, weaker choice principles that would suffice instead?
Can the exact amount of choice necessary be specified? These
questions will start to be addressed in the forthcoming
\cite{LRi}, but there is certainly more that can be done than is
even attempted there.

\end{document}